\documentclass[a4,10pt]{amsart}
\usepackage{amsfonts}

\usepackage{amscd}
\usepackage{amsmath}
\usepackage{amssymb,amsthm}
\usepackage[mathscr]{eucal}

\newtheorem{theo}{Theorem}[section]
\newtheorem{prop}{Proposition}
\newtheorem{lem}[theo]{Lemma}

\theoremstyle{definition}
\newtheorem{definition}{Definition}

\newfont{\got}{eufm10}

\begin{document}
\title{A Mahler measure of a K3-hypersurface expressed as a Dirichlet L-series }
\author{Marie Jos\' e Bertin}
\date{\today}

\keywords{Modular Mahler measure,
Eisenstein-Kronecker's Series, $L$-series of $K3$-surfaces, $l$-adic representations, Livn\'e Criterion, Rankin-Cohen brackets}

\subjclass{11, 14D, 14J}
\email{bertin@math.jussieu.fr}
\address{Universit\' e Pierre et Marie Curie (Paris 6), Institut de Math\'ematiques, 175 rue du Chevaleret, 75013 PARIS }

\maketitle

\begin{abstract} 

We present another example of a $3$-variable polynomial defining a K3-hypersurface and having a logarithmic Mahler measure expressed in terms of a Dirichlet L-series.

\end{abstract}

\section{ Introduction}

The logarithmic Mahler measure $m(P)$ of a Laurent polynomial $P\in \mathbb{C}[X_1^{\pm} ,...,X_n^{\pm} ]$ is defined by

$$
m(P)=\frac {1}{(2\pi i)^n} \int_{\mathbb{T} ^n}\log \vert P(x_1^\pm
,...,x_n^\pm)\vert\frac{dx_1}{x_1}...\frac{dx_n}{x_n}
$$
where $\mathbb{T}^n$ is the n-torus $\{(x_1,...,x_n)\in \mathbb{C}^n /
\vert x_1\vert =...=\vert x_n\vert =1\}$. 

For $n=2$ and polynomials $P$ defining elliptic curves $E$, conjectures have been made, with proofs in the $CM$ case, by various authors \cite{Bo2}, \cite{RV1}, \cite{RV2}. These conjectures give conditions on the polynomial $P$ for getting explicit expressions of $m(P)$ in terms of the $L$-series of $E$. A crucial condition for $P$ is to be ``tempered'', that is the roots of the polynomials of the faces of its Newton polygon are only roots of unity. This condition is related to the link between $m(P)$ and the second group of $K$-theory, \cite{Bei}, \cite{RV2}.

In various papers we obtained results for $n=3$ and polynomials $P$ defining $K3$-surfaces, \cite{Ber1}, \cite{Ber2}, \cite{Ber3}. Our aim is to find an analog of the previous results for $K3$-surfaces. In particular, which condition on the polynomial $P$ ensure the expression of $m(P)$ in terms of the $L$-series of the $K3$-surface plus a Dirichlet $L$-series?
Our investigations concern two families of polynomials in three variables \cite{Ber1}.

This result is the second example of a Mahler measure expressed uniquely in terms of a Dirichlet $L$-series.

The first example  was
$$m(P_0)=m(X+\frac {1}{X}+Y+\frac {1}{Y}+Z+\frac {1}{Z})=d_3=\frac {3\sqrt{3}}{4\pi}L(\chi_{-3},2),$$
where $L(\chi_{-3},2)$ denotes the Dirichlet $L$-series for the quadratic character $\chi_{-3}$ attached to the imaginary quadratic field $\mathbb Q(\sqrt{-3})$.
This equality is easy to prove since the modular part, I mean the part corresponding to the $L$-series of the $K3$-surface, is obviously $0$.

The second example is the following theorem.

\begin{theo}
Let $Q_{-3}$ the Laurent polynomial
$$
\begin{aligned}
Q_{-3} & =X+\frac {1}{X}+Y+\frac {1}{Y}+Z+\frac {1}{Z} \\
        & +XY+\frac {1}{XY}+ZY+\frac {1}{ZY}+XYZ+\frac {1}{XYZ}+3
\end{aligned}
$$

and define
$$d_3=\frac {3\sqrt{3}}{4\pi}L(\chi_{-3},2).$$
Then
$$m(Q_{-3})=\frac {8}{5}d_3.$$
\end{theo}

\bigskip

In this theorem the evaluation of the modular part needs the use of Livn\'e's criterion \cite{Y}, since we have to compare two $l$-adic representations, and also recent results about Dirichlet $L$-series \cite{Zu1}.

\bigskip

\noindent{\textbf{Acknowledgments}}

The measure $m(Q_{-3})$ was guessed numerically some years ago by Boyd \cite{Bo1}. His guess and some discussions with Zagier \cite{Za} were probably determinant for the discovery of the proof. So I am pleased to address my grateful thanks to both of them.

\section{Some facts}

The polynomial $Q_{-3}$ belong to the family of polynomials $Q_k$ whose Mahler measure has been studied in a previous paper \cite{Ber1}.

\begin{theo}
Consider the family of Laurent polynomials
 
$$
\begin{aligned}
Q_{k} & =X+\frac {1}{X}+Y+\frac {1}{Y}+Z+\frac {1}{Z} \\
        & +XY+\frac {1}{XY}+ZY+\frac {1}{ZY}+XYZ+\frac {1}{XYZ}-k.
\end{aligned}
$$

Let $k=-(t+\frac {1}{t})-2$ and define
$$
t=\frac {\eta (3\tau)^4\eta (12\tau)^8\eta(2\tau)^{12}}{\eta(\tau)^4\eta
(4\tau)^8\eta (6\tau)^{12}},
$$
where $\eta$ denotes the Dedekind eta function 
$$
\eta (\tau)=e^{\frac {\pi i \tau}{12}} \prod_{n\geq 1}(1-e^{2\pi i n\tau}).
$$
 
Then 

$$
\begin{aligned}  m(Q_k)=&\frac {\Im \tau}{8 \pi ^3}\{ \sum'_{m,\kappa}(2(2\Re
\frac {1}{(m\tau+\kappa)^3(m\bar{\tau}+\kappa)}+\frac
{1}{(m\tau+\kappa)^2(m\bar{\tau}+\kappa)^2})\\ &-32(2\Re \frac
{1}{(2m\tau+\kappa)^3(2m\bar{\tau}+\kappa)}+\frac
{1}{(2m\tau+\kappa)^2(2m\bar{\tau}+\kappa)^2})\\ &-18(2\Re \frac
{1}{(3m\tau+\kappa)^3(3m\bar{\tau}+\kappa)}+\frac
{1}{(3m\tau+\kappa)^2(3m\bar{\tau}+\kappa)^2})\\ &+288(2\Re \frac
{1}{(6m\tau+\kappa)^3(6m\bar{\tau}+\kappa)}+\frac
{1}{(6m\tau+\kappa)^2(6m\bar{\tau}+\kappa)^2}))\}  \end{aligned} $$

\end{theo}

\bigskip

\bigskip

Let us recall now the following results.

Given a normalised Hecke eigenform $f$ of some level $N$ and weight $k=3$, we can associate a Galois representation \cite{Hu}, \cite{Sh}
$$\rho_f: {\text Gal}(\bar{\mathbb Q}/\mathbb Q) \rightarrow {\text Gl}(2,\mathbb Q_l).$$

To a normalised Hecke newform $f$ can also be associated an $L$-function $L(f,s)$ by
$$L(f,s):=L(\rho_f,s)$$
(the $L$-series of the Galois representation $\rho_f$). Equivalently, if $f$ has a Fourier expansion $f=\sum_n b_n q^n$, then $L(f,s)$ is also the Mellin transform of $f$
$$L(f,s)=\sum_n \frac {b_n}{n^s}.$$
Moreover, the series $L(f,s)$ has a product expansion
$$L(f,s)=\sum_{n\geq 1} \frac {b_n}{n^s}=\prod_p \frac {1}{1-b_pp^{-s}+\chi(p)p^{k-1-2s}}$$
where $\chi(p)=0$ if $p\mid N$.

\bigskip
Concerning the comparison between $l$-adic representations, Serre's then Livn\'e's result can be found for example in \cite{Y}, \cite{PT}.

\bigskip  
\begin{lem}
Let $\rho_l,\rho_l':G_{\mathbb Q} \rightarrow \hbox{Aut}V_l$ two rational $l$-adic representations with $\hbox{Tr}F_{p,\rho_l}=\hbox{Tr}F_{p,\rho_l'}$ for a set of primes $p$ of density one (i.e. for all but finitely many primes). If $\rho_l$ and $\rho_l'$ fit into two strictly compatible systems, the $L$-functions associated to these systems are the same.
\end{lem}

Then the great idea (Serre \cite{Se} , Livn\'e \cite{Li}) is to replace this set of primes of density one by a finite set.

\begin{definition}
A finite set $T$ of primes is said to be an effective test set for a rational Galois representation $\rho_l :G_{\mathbb Q} \rightarrow \hbox{Aut} V_l$ if the previous lemma holds with the set of density one replaced by $T$.
\end{definition}

\begin{definition}
Let $\mathcal P$ denote the set of primes, $S$ a finite subset of $\mathcal P$ with $r$ elements, $S'=S\cup \{-1 \}$. Define for each $t\in \mathcal P$, $t\neq 2$ and each $s\in S'$ the function
$$f_s(t):=\frac {1}{2}(1+\left( \frac {s}{t}\right))$$
and if $T \subset \mathcal P$, $T\cap S=\emptyset$,
$$f:T\rightarrow \left(\mathbb Z/2\mathbb Z \right)^{r+1}$$
such that $$f(t)=\left( f_s(t) \right)_{s\in S'}$$.
\end{definition}

\begin{theo} (Livn\'e's criterion) Let $\rho$ and $\rho'$ be two $2$-adic $G_{\mathbb Q}$-representations which are unramified outside a finite set $S$ of primes, satisfying
$$\hbox{Tr} F_{p,\rho} \equiv \hbox{Tr} F_{p,\rho'} \equiv 0 (\hbox{mod} 2)$$
and
$$\hbox{det} F_{p,\rho} \equiv \hbox{det} F_{p,\rho'} (\hbox{mod} 2)$$
for all $p\notin S \cup \{2\}$.

Any finite set $T$ of rational primes disjoint from $S$ with $f(T)=\left ( \mathbb Z/2\mathbb Z \right )^{r+1}\backslash \{0\}$ is an effective test set for $\rho$ with respect to $\rho'$.
\end{theo}

\bigskip

The $K3$-surface $\tilde{X}$ defined by the polynomial $Q_{-3}$ has been studied by Peters, Top and van der Vlugt \cite{PT}. In particular they proved the theorem.

\begin{theo}
There is a system $\rho=(\rho_l)$ of 2-dimensional $l$-adic representations of $G_{\mathbb Q}=\text{Gal}(\bar{\mathbb Q}/\mathbb Q)$
$$\rho_l: G_{\mathbb Q} \rightarrow \text{Aut}H_{\text{trc}}^2(\tilde{X},\mathbb Q_l).$$
The system $\rho=(\rho_l)$ has an $L$-function
$$L(s,\rho)=\prod_{p\neq 3,5} \frac {1}{1-A_p p^{-s}+\left(\frac {p}{15}\right)p^2p^{-2s}}.$$

This $L$-function is the $L$-function of the modular form $f^+=g\theta_1 \in S_3(15,\left(\frac {.}{15}\right))$ where 
$$\theta_1=\sum_{m,n\in \mathbb Z}q^{m^2+mn+4n^2} \,\,\,\,\,\,\,g=\eta(z) \eta(3z) \eta(5z) \eta(15z)$$
and $\eta$ is the Dedekind eta function. The Mellin transform  $\sum \frac {b_n}{n^s}$ of $f^+$ satisfies $b_p=A_p$ for $p\neq 3,5$, where $A_p$ can be computed as follows.
\begin{itemize}
\item If $p\equiv 1$ or $4$ mod. $15$, find an integral solution of the equation $x^2+xy+4y^2=p$. Then $A_p=2x^2-7y^2+2xy$.
\item If $p\equiv 2$ or $8$ mod. $15$, find an integral solution of the equation $2x^2+xy+2y^2=p$. Then $A_p=x^2+8xy+y^2$.
\end{itemize}

\end{theo}

\section{Proof of theorem 1}

The proof follows from three propositions.

\begin{prop}

$$ 
\begin{aligned}
m(Q_{-3})= & \frac{3\sqrt{15}}{\pi^3}\sum'_{m',\kappa}\left ( \frac {15k^2-m'^2}{(m'^2+15\kappa^2)^3}+ \frac {-5k^2+3m'^2}{(3m'^2+5\kappa^2)^3}\right)\\
 & +\left(\frac {1}{2}\frac{2m'^2+2m'\kappa-7\kappa^2}{(m'^2+m'\kappa +4\kappa^2)^3}+\frac {1}{2}\frac{m'^2+8m'\kappa+\kappa^2}{(2m'^2+m'\kappa +2\kappa^2)^3} \right )\\
 &+ \frac{6\sqrt{15}}{\pi^3}\sum'_{m',\kappa}\left (\frac {1}{(m'^2+15\kappa^2)^2} -\frac {1}{(3m'^2+5\kappa^2)^2}\right)\\
 &+\left(\frac{1}{(2m'^2+m'\kappa +2\kappa^2)^2} -\frac{1}{(m'^2+m'\kappa +4\kappa^2)^2}\right )
\end{aligned}
$$

\end{prop}

\begin{proof}

Define
$$D_{j\tau}=(mj\tau+\kappa)(mj\bar {\tau}+\kappa).$$
So
$$\begin{aligned}
m(Q_k)=\frac {\Im \tau}{8\pi^3}\sum'_{m,\kappa} & [2\frac
{(m(\tau+\bar {\tau})+2\kappa)^2}{D_{\tau}^3}+\frac {-2}{D_{\tau}^2}\\
                                             & -32\frac
{(2m(\tau+\bar {\tau})+2\kappa)^2}{D_{2\tau}^3}+\frac
{32}{D_{2\tau}^2}\\
                                             & -18\frac
{(3m(\tau+\bar {\tau})+2\kappa)^2}{D_{3\tau}^3}+\frac
{18}{D_{3\tau}^2}\\
                                             & +288\frac
{(6m(\tau+\bar {\tau})+2\kappa)^2}{D_{6\tau}^3}-\frac
{288}{D_{6\tau}^2}]
\end{aligned}$$ 
If $k=-3$, then $\tau=\frac {-3+\sqrt{-15}}{24}$ and 

$$\begin{aligned}
&D_\tau=\frac {1}{24}(m^2-6m\kappa +24\kappa^2)=\frac {1}{24}(m'^2+15\kappa^2)\,\,\,\,\,{\hbox{with}}\,\,\,\,\,m'=m-3\kappa\\
& \\
 & D_{2\tau}=\frac {1}{6}(m^2-3m\kappa +6\kappa^2)=\frac {1}{6}(m'^2+m'\kappa+4\kappa^2)\,\,\,\,\,{\hbox{with}}\,\,\,\,\,m'=m-2\kappa\\
 & \\
& D_{3\tau}=\frac {1}{8}(3m^2-6m\kappa +8\kappa^2)=\frac {1}{8}(3m'^2+5\kappa^2)\,\,\,\,\,{\hbox{with}}\,\,\,\,\,m'=m-\kappa\\
 & \\
& D_{6\tau}=\frac {1}{2}(3m^2-3m\kappa +2\kappa^2)=\frac {1}{2}(2m^2+m\kappa+2\kappa'^2)\,\,\,\,\,{\hbox{with}}\,\,\,\,\,\kappa'=\kappa-m.
\end{aligned}$$
Thus
$$m(Q_{-3})=\frac {\sqrt{15}}{24\times 8 \pi^3}\sum'_{m',\kappa}(A_1+A_2+A_3+A_4).$$
Now $A_1$ can be written
$$A_1=(24)^2 \left ( \frac{-m'^2+15\kappa^2-30m'\kappa}{(m'^2+15\kappa^2)^3}+\frac {2}{(m'^2+15\kappa^2)^2} \right )$$
and 
$$\sum'_{m',\kappa}A_1=(24)^2\sum'_{m',\kappa}\left (\frac {15k^2-m'^2}{(m'^2+15\kappa^2)^3}+\frac {2}{(m'^2+15\kappa^2)^2}\right ).$$

Then, we get 
$$A_2=(24)^2 \left ( \frac{m'^2+16m'\kappa+4\kappa^2}{(m'^2+m'\kappa +4\kappa^2)^3}-\frac {2}{(m'^2+m'\kappa+4\kappa^2)^2} \right )$$
Now with the change of variable $\kappa=\kappa'-m'$ we put the denominators of $A_2$ symmetric with respect to $m'$ and $\kappa'$.
So

$$A_2=(24)^2 \left ( \frac{-11m'^2+8m'\kappa'+4\kappa'^2}{(4m'^2-7m'\kappa' +4\kappa'^2)^3}-\frac {2}{(4m'^2-7m'\kappa'+4\kappa^2)^2} \right )$$
that is
$$A_2=(24)^2 \left ( \frac {1}{2}\frac{-7m'^2+16m'\kappa'-7\kappa'^2}{(4m'^2-7m'\kappa' +4\kappa'^2)^3}-\frac {2}{(4m'^2-7m'\kappa'+4\kappa^2)^2} \right )$$
and coming back to variables $m'$ and $\kappa$,

$$A_2=(24)^2 \left ( \frac {1}{2}\frac{2m'^2+2m'\kappa-7\kappa^2}{(m'^2+m'\kappa +4\kappa^2)^3}-\frac {2}{(m'^2+m'\kappa+4\kappa^2)^2} \right ).$$
The same way we obtain,
$$A_3=(24)^2 \left ( \frac{3m'^2+30m'\kappa-5\kappa^2}{(3m'^2+5\kappa^2)^3}-\frac {2}{(3m'^2+5\kappa^2)^2} \right )$$
or
$$A_3=(24)^2 \left ( \frac{3m'^2-5\kappa^2}{(3m'^2+5\kappa^2)^3}-\frac {2}{(3m'^2+5\kappa^2)^2} \right ).$$
Finally using the same tricks as for $A_2$, we obtain
$$A_4=(24)^2 \left ( \frac {1}{2}\frac{m^2+8m\kappa'+\kappa'^2}{(2m^2+m\kappa' +2\kappa'^2)^3}+\frac {2}{(2m^2+m\kappa'+2\kappa'^2)^2} \right ).$$
\end{proof}

\bigskip

From proposition 1. we notice that the Mahler measure is expressed as a sum of a modular part

$$ 
\begin{aligned}
 & \frac{3\sqrt{15}}{\pi^3}\sum'_{m',\kappa}\left ( \frac {15k^2-m'^2}{(m'^2+15\kappa^2)^3}+ \frac {-5k^2+3m'^2}{(3m'^2+5\kappa^2)^3}\right)\\
 & +\left(\frac {1}{2}\frac{2m'^2+2m'\kappa-7\kappa^2}{(m'^2+m'\kappa +4\kappa^2)^3}+\frac {1}{2}\frac{m'^2+8m'\kappa+\kappa^2}{(2m'^2+m'\kappa +2\kappa^2)^3} \right )
\end{aligned}
$$
and a part related to a Dirichlet $L$-series

$$ 
\begin{aligned}
 &+ \frac{6\sqrt{15}}{\pi^3}\sum'_{m',\kappa}\left (\frac {1}{(m'^2+15\kappa^2)^2} -\frac {1}{(3m'^2+5\kappa^2)^2}\right)\\
 &+\left(\frac{1}{(2m'^2+m'\kappa +2\kappa^2)^2} -\frac{1}{(m'^2+m'\kappa +4\kappa^2)^2}\right ).
\end{aligned}
$$

To prove that the modular part is $0$, we observe first that 

$$L(f_1,s)=\frac {1}{2}\sum'_{r,s}\frac {5r^2-3k^2}{(3r^2+5k^2)^s}\,\,\,\,{\hbox{and}}\,\,\,\,L(f_2,s)=\frac {1}{2}\sum'_{r,s}\frac {r^2-15k^2}{(r^2+15k^2)^s}$$

are the Mellin transform of the two weight $3$ modular forms

$$f_1=\frac {1}{2}\sum_{r,s\in \mathbb Z}(5r^2-3k^2)q^{3r^2+5k^2}\,\,\,\,\,\,\,\,f_2=\frac {1}{2}\sum_{r,s\in \mathbb Z}(r^2-15k^2)q^{r^2+15k^2}.$$

Then using theorem $2.4$ we know that

$$\sum'\left(\frac {1}{4}\frac{2m'^2+2m'\kappa-7\kappa^2}{(m'^2+m'\kappa +4\kappa^2)^s}+\frac {1}{4}\frac{m^2+8m\kappa'+\kappa'^2}{(2m^2+m\kappa' +2\kappa'^2)^s} \right )=L(f^+,s)$$
is the $L$-series attached to the modular $K3$-surface $\tilde{X}$.

\begin{prop}
$$
\begin{aligned}
\sum'_{m,k}\left ( \frac {-15k^2+m^2}{(m^2+15k^2)^3}+ \frac {5k^2-3m^2}{(3m^2+5k^2)^3}\right)=\\
\sum'_{m,k}\left(\frac {1}{2}\frac{2m^2+2mk-7k^2}{(m^2+mk +4k^2)^3}+\frac {1}{2}\frac{m^2+8mk+k^2}{(2m^2+mk +2k^2)^3} \right ).
\end{aligned}
$$

\end{prop}

\bigskip

\begin{proof}
Let $a$ a rational integer and denote $\theta_a=\sum_{n\in \mathbb Z}q^{an^2}$ the weight $1/2$ modular form for the congruence group $\Gamma=\Gamma_0(4)$. Denote
$$f_1:=[\theta_5,\theta_3]\,\,\,\,\,\,\,\,\,\,f_2:=[\theta_1,\theta_{15}]$$
the Rankin-Cohen brackets wich are modular forms of weight $3$ for $\Gamma$.

Recall that, if $f$ and $g$ are modular forms of respective weight $k$ and $l$ for a congruence subgroup, then its Rankin-Cohen bracket is the modular form of weight $k+l+2$ defined by
$$[g,h]:=kgh'-lg'h.$$

Thus we get the two weight $3$ modular forms
$$f_1=\frac {1}{2}\sum_{r,s\in \mathbb Z}(5r^2-3k^2)q^{3r^2+5k^2}\,\,\,\,\,\,\,\,f_2=\frac {1}{2}\sum_{r,s\in \mathbb Z}(r^2-15k^2)q^{r^2+15k^2}.$$

So to compare $L(f_1,s)+L(f_2,s)=\sum\frac {A_1(n)}{n^s}$ and $L(f^+,s)=\sum\frac {A_2(n)}{n^s}$ we apply Livn\'e's criterion.

First we determine an effective test set $T$ for the respective representations
$$T=\{7,11,13,17,19,23,29,31,41,43,53,61,71,73,83\}.$$

Then we compute the corresponding $A_1(p)$ and $A_2(p)$.

\begin{center}
\begin{tabular}{|l||r|r|r|r|r|r|r|r|r|r|r|r|r|r|r|}
\hline

p & 7 & 11 & 13 & 17 & 19 & 23 & 29 & 31 & 41 & 43 & 53 & 61 & 71 &73 & 83 \\ \hline
$A_1(p)$ & 0 & 0 & 0 & -14 & -22 & 34 & 0 & 2 & 0 & 0 & -86 & -118 & 0 & 0 & 154\\ \hline
$A_2(p)$ & 0 & 0 & 0 & -14 & -22 & 34 & 0 & 2 & 0 & 0 & -86 & -118 & 0 & 0 & 154\\ \hline
\end{tabular}
\end{center}

This achieves the proof of the proposition.

\end{proof}

\begin{prop}
$$
\begin{aligned}
\frac{6\sqrt{15}}{\pi^3}\sum'_{m,k} & \frac {1}{(m^2+15k^2)^2} -\frac {1}{(3m^2+5k^2)^2}\\
 & +\frac{1}{(2m^2+mk +2k^2)^2}-\frac{1}{(m^2+mk +4k^2)^2}\\
 &  =\frac {8}{5}d_3
\end{aligned}
$$
\end{prop}

\begin{proof}
We denote
$$L_f(s):=L(\chi _f,s)$$
the Dirichlet's $L$-series for the character $\chi_f$ attached to the quadratic field $\mathbb Q(\sqrt{f})$.

The proof follows from a lemma.

\begin{lem}

\begin{enumerate}

\item $$\sum'_{m,k}\left ( \frac {1}{(2m^2+mk+k^2)^s}+\frac {1}{m^2+mk+4k^2)^s}\right )=2\zeta(s)L_{-15}(s)$$

\item $$\sum'_{m,k} \left (\frac {1}{(3m^2+5k^2)^s}+\frac {1}{(m^2+15k^2)^s} \right ) =2(1+\frac {1}{2^{2s-1}}-\frac {1}{2^{s-1}}) \zeta(s)L_{-15}(s)$$

\item $$\sum'_{m,k}\left ( \frac {1}{(m^2+mk+4k^2)^s}-\frac {1}{2m^2+mk+2k^2)^s}\right )=2L_{-3}(s)L_5(s)$$

\item $$\sum'_{m,k} \left (\frac {1}{(m^2+15k^2)^s}-\frac {1}{(3m^2+5k^2)^s} \right ) =2(1+\frac {1}{2^{2s-1}}+\frac {1}{2^{s-1}}) L_{-3}(s)L_{5}(s)$$

\end{enumerate}

\end{lem}

\bigskip

\begin{proof}
The assertion (1) follows from the result \cite{ZG}
$$\sum'\left ( \frac {1}{(2m^2+mk+k^2)^s}+\frac {1}{m^2+mk+4k^2)^s}\right )=\zeta_{\mathbb Q(\sqrt{-15})}(s)$$
and the formula
$$\zeta_{\mathbb Q(\sqrt{-15})}(s)=\zeta(s)L_{-15}(s).$$
The assertion (2) follows from results of K. Williams \cite{Wi} and Zucker \cite{Zu1}.
Taking Williams's notations we set

$$\phi(q):=\sum_{-\infty}^{+\infty}q^{n^2}$$
and get

$$\phi(q)\phi(q^{15})+\phi(q^3)\phi(q^5)=2+\sum_{n\geq 1}a_n(-60)\frac {q^n}{1-q^n},$$
where
$$a_n(-60)=
\begin{cases}
0 & \text{if $n\equiv 0,3,5,6,9,10,$ (mod. $60$)}\\
2   & \text{if $n\equiv 1,4,8,14,16,17,19,22,23,26,31,32,47,49,53,58 $(mod. $60$)}\\
-2  & \text{if $n\equiv 2,7,11,13,28,29,34,37,38,41,43,44,46,52,56,59$(mod. $60$)}.
\end{cases}
$$

As explained in \cite{Zu1}, often we may get 
$$Q(a,b,c;s)=\sum' \frac {1}{(am^2+bmn+cn^2)^s}$$
in terms of $L_{\pm h}$ when expressing them as Mellin transforms of products of various Jacobi functions $\theta_3(q)$ for different arguments.

More precisely,

$$
\begin{aligned}
Q(1,0,\lambda; s ) & =\frac {1}{\Gamma(s)}\int_0^{\infty} t^{s-1}\sum' e^{-(m^2t+\lambda n^2t)}dt\\
                & =\frac {1}{\Gamma(s)}\int_0^{\infty} (\theta_3(q)\theta_3(q^{\lambda})-1)dt
\end{aligned}
$$

where $e^{-t}=q$ and
$$\theta_3(q)=1+2q^2+2q^4+2q^9+\hdots;$$
thus writing $\theta_3(q)\theta_3(q^{\lambda})-1$ as a Lambert series $\sum_{n\geq 1} a_n\frac {q^n}{1-q^n}$, very often the integral is given in terms of $L$-series.

So we get

$$
\begin{aligned}
Q(1,0,15;s)+Q(3,0,5;s) & = \frac {1}{\Gamma (s)}\int_0^{\infty} t^{s-1}(\theta_3(q)\theta_3(q^{15})+\theta_3(q^3)\theta_3(q^5)-2)dt\\
                       & =\frac {1}{\Gamma(s)}\int_0^{+\infty}t^{s-1}(\sum_{n\geq 1}a_n(-60)\frac {e^{-tn}}{1-e^{-tn}})dt.
\end{aligned}
$$

Since
$$\Gamma(s)=\int_0^{+\infty} e^{-y}y^{s-1}dy$$
making the change variable $nt=y$, it follows

$$
\begin{aligned}
\frac {1}{\Gamma(s)}\int_0^{+\infty}t^{s-1}\frac {e^{-tn}}{1-e^{-tn}}dt & = \int_0^{+\infty}\left( \frac {y}{n}\right)^{s-1}\frac {e^{-y}}{1-e^{-y}}\frac {dy}{n}\\ 
                                                      & =\frac  {1}{\Gamma(s)}\frac{1}{n^s}\int_0^{+\infty}\frac {y^{s-1}}{e^y-1}dy\\ 
                                             & =\frac {1}{n^s} \zeta(s).
\end{aligned}
$$

Thus
$$Q(1,0,15;s)+Q(3,0,5;s)=\zeta(s)\sum_{n\geq 1}a_n(-60)\frac {1}{n^s}.$$
But

$$
\begin{aligned}
L_{-60}(s) & = \frac {1}{1^s}- \frac {1}{7^s}- \frac {1}{11^s}- \frac {1}{13^s}+ \frac {1}{17^s}+ \frac {1}{19^s}+ \frac {1}{23^s}+ \frac {1}{31^s}\\
            & -\frac {1}{37^s}- \frac {1}{41^s}- \frac {1}{43^s}+\frac {1}{47^s}+ \frac {1}{49^s}+ \frac {1}{53^s}- \frac {1}{59^s}+ \hdots \text{ (mod. $60$)}
\end{aligned}
$$

and 

$$
\begin{aligned}
L_{-15}(s)= & \frac {1}{1^s}+\frac {1}{2^s}+ \frac {1}{4^s}- \frac {1}{7^s}+ \frac {1}{8^s}- \frac {1}{11^s}- \frac {1}{13^s}\\
            & -\frac {1}{14^s}+ \frac {1}{16^s}+ \frac {1}{17^s}+\frac {1}{19^s}- \frac {1}{22^s}+ \hdots \text{ (mod. $15$)}.
\end{aligned}
$$

So,

$$
\begin{aligned}
& \frac {1}{2} \sum_{n\geq 1}a_n(-60)\frac {1}{n^s} = L_{-60}(s)+\frac {1}{2^s}(-1+ \frac {1}{2^s}+ \frac {1}{4^s}+ \frac {1}{7^s}+ \frac {1}{8^s}+ \frac {1}{11^s}+ \frac {1}{13^s}- \frac {1}{14^s}\\
      &    +\frac {1}{16^s}- \frac {1}{17^s}- \frac {1}{19^s}-\frac {1}{22^s}- \frac {1}{23^s} -\frac {1}{26^s}- \frac {1}{28^s}+ \frac {1}{29^s}+ \hdots) \text{ (mod. $30$)}.          
\end{aligned}
$$

Let us define
$$L_{-15}(s):=\sum_{n\geq 1}\frac {\chi_{-15}(n)}{n^s}=L_{+}(s)+L_{-}(s)$$
where
$$L_{+}(s)=\sum_{n\geq 1, \,\,\,n\,\,\, \text{pair}}\frac {\chi_{-15}(n)}{n^s}\,\,\,\,\,\,\,\,\,\,\,\,L_{-}(s)=\sum_{n\geq 1,\,\,\,n\,\,\, \text{impair}}\frac {\chi_{-15}(n)}{n^s}.$$

Obviously,
$$L_{+}(s)=\frac {1}{2^s}L_{-15}(s),\,\,\,\,\,L_{-60}(s)=L_{-}(s),\,\,\,\,\,\,L_{-15}(s)=L_{-}(s)+\frac {1}{2^s}L_{-15}(s).$$

Thus,

$$
\begin{aligned}
\frac {1}{2}\sum_{n\geq 1}\frac {a_n(-60)}{n^s} & =L_{-}(s)+\frac {1}{2^s}(L_{+}(s)-L_{-}(s))\\
                          & = (1+\frac {1}{2^{2s-1}}-\frac {1}{2^{s-1}})L_{-15}(s).
\end{aligned}
$$

From this last equality we deduce the formula (2).

From \cite{Zu-Ro} we get
$$Q(1,1,4;s)=\zeta(s)L_{-15}(s)+L_{-3}(s)L_{5}(s)$$
so from formula (1) we obtain the formula (3).

Equality (4) derives from a formula by Zucker and Robertson \cite{Zu-Ro} giving

$$
\begin{aligned}
Q(1,0,15;s) & =(1-\frac {1}{2^{s-1}}+\frac {1}{2^{2s-1}})\zeta(s)L_{-15}(s)\\
       &  + (1+\frac {1}{2^{s-1}}+\frac {1}{2^{2s-1}})L_{-3}(s)L_{5}(s).
\end{aligned}
$$

So, thanks to formula (2)

$$
\begin{aligned}
Q(1,0,15;s)-Q(3,0,5;s) & =2Q(1,0,15;s)-(Q(1,0,15;s)+Q(3,0,5;s))\\
         & =2(1+\frac {1}{2^{s-1}}+\frac {1}{2^{2s-1}})L_{-3}(s)L_{5}(s)
\end{aligned}
$$

\end{proof}
By substracting (3) to (4) for $s=2$ and using \cite{Zu1}
$$L_5=\frac {4\pi^2}{25\sqrt{5}},$$
we get the proposition.

\end{proof}

\bigskip

The proof of theorem 1.1 is just a combination of the three propositions.

\bigskip
\bigskip

\end{document}